\documentclass[12pt]{amsart}
\usepackage{amssymb}
\usepackage{color}
\usepackage{fullpage}
\usepackage{graphicx}  
\newtheorem{theorem}{Theorem}[section]\newtheorem{definition}[theorem]{Definition}\newtheorem{lemma}[theorem]{Lemma}
\newtheorem{proposition}[theorem]{Proposition}

\newtheorem{corollary}[theorem]{Corollary}  

\begin{document}
\title{Leinert Sets and Complemented Ideals in Fourier Algebras}
\author{Michael Brannan, Brian Forrest and Cameron Zwarich}
\address{Michael Brannan: Department of Mathematics, Texas A\&M University, Mailstop 3368, College Station, TX 77843-3368 USA}
\email{mbrannan@math.tamu.edu}
\address{Brian Forrest: Department of Pure Mathematics, University of Waterloo, Waterloo, Ontario, N2L 3G1 Canada}
\email{beforrest@uwaterloo.ca}
\address{Cameron Zwarich: Cupertino, California,  95014 USA}
\email{cwzwarich@me.com}
\date{\today}
\maketitle

\begin{abstract}
Let $G$ be a locally compact group. We show how complemented ideals in the Fourier algebra $A(G)$ of $G$ arise naturally 
from a class of thin sets known as Leinert sets. Moreover, we also present an explicit example of a closed ideal in $A(\mathbb{F}_{N})$,
the free group on $N \ge 2$ generators,  that is complemented in $A(\mathbb{F}_{N})$ but it is not 
completely complemented. Then by establishing an appropriate extension result for restriction algebras 
arising from Leinert sets, we show that any almost connected group $G$ for which every complemented ideal in $A(G)$ is also completely complemented must be amenable. 
\end{abstract}

\section{Introduction} 
The problem of identifying complemented ideals in the Fourier algebra $A(G)$ of a locally compact group began naturally with an attempt to characterize 
such ideals in the group algebra $L^1(G)$ of a locally compact abelian group. (In this case, $A(G)$ can be identified with $L^1(\hat{G})$ via the classical Fourier transform.) The first significant result in this direction can be attributed to D.J. Newman  
\cite{Newman}. He showed that if $\mathbb T$ is the circle group and if 
\[H^1=\{f \in L^1(\mathbb T)~|~ \hat{f}(n)=0 \textnormal{~for every~} n<0\},\]
then $H^1$ is not complemented in $L^1(\mathbb T)$. A year later, W. Rudin gave a complete characterization of all of the complemented
 ideals in $L^1(\mathbb T)$ by showing that an ideal $I$ is complemented in $L^1(\mathbb T)$ if and only if 
\[I=\{f \in L^1(\mathbb T)~|~ \hat{f}(n)=0 \textnormal{~for every~} n\in E\},\]
where $E\subset \mathbb{Z}$ is of the form $E=\bigcup_{i=1}^n a_i \mathbb{Z} + b_i$. 

Given a locally compact group $G$ we let $G_d$ denote the same group together with the discrete topology. We will denote the ring of subsets of $G_d$ generated by left cosets of subgroups of $G$ by $\mathcal{R}(G_d)$. The closed coset ring of $G$ is the collection 
\[ \mathcal{R}_c(G)=\{E \subseteq G~|~ E\in \mathcal{R}(G_d) \textnormal{~and~} E\textnormal{~is closed in ~}G\}.\]

In his beautiful 1966 memoir \cite{Rosenthal}, H.P. Rosenthal showed that for a general locally compact abelian group $G$ a 
necessary condition for a closed ideal $I$ to be complemented in $L^1(G)$ is that $I=I_G(E)$ where 
\[I_G(E)= \{ f \in L^1(G) ~|~ \hat{f}(x)=0 \textnormal{~for every~} x\in E\}\]
with $E\in \mathcal{R}_c(\hat{G})$.   In fact, Rudin's result above established the converse for the specific case where 
$G=\mathbb T$. However the converse does not hold in general. In particular, Alspach and Matheson \cite{Al-Math} showed that for
 $G=\mathbb{R}$, I is complemented in $L^1(\mathbb{R})$ if and only if $I=I_G(E)$ where
 $E=\bigcup_{i=1}^n (\tau_i \mathbb{Z} + \beta_i)\setminus F$ where the $t_i$'s are pairwise rationally dependent and $F$ is finite. From there Alspach, Matheson and Rosenblatt made a systematic study of complemented ideals in the group algebra of a generic locally compact abelian group 
\cite{Al-Math-Rosen}, \cite{Al-Math-Rosen2}.  They gave a necessary and sufficient condition for an ideal with a discrete hull to be complemented. They also developed a complicated inductive procedure which Alspach later used to completely characterize the complemented ideals in $L^1(\mathbb{R}^2)$
\cite{Alsp}. Their analysis also showed that a complete characterization of the complemented ideals in $L^1(\mathbb{R}^3)$ would be a monumental task.  

For nonabelian groups the problem of identifying complemented ideals in the Fourier algebra is  still very much in it's infancy.  If we recognize the additional structure that $A(G)$ carries as an operator space, and we ask that our projection be completely bounded, then it is known from \cite{F-K-L-S} that if $G$ is amenable, then the  ideal must be of the form $I=I_G(E)$ where  $E\in \mathcal{R}_c({G})$. However as we will also make clear later in the paper, the necessity that $E\in \mathcal{R}_c({G})$ for $I=I_G(E)$ to be completely complemented does not hold for $\mathbb{F}_N$, the free group on $N\ge 2$ generators. Moreover, we will show that for certain non-amenable almost connected groups $G$, there exists a set $E \not \in \mathcal{R}_c({G})$ such that $I_G(E)$  is completely complemented in $A(G)$ (see Proposition \ref{prop-cb-proj}). 

Unfortunately, even for an amenable group $G$, without the additional assumption that our projection be completely bounded, it is not known whether or not the complementation of  $I_G(E)$ in $A(G)$ implies that $E\in \mathcal{R}_c({G})$. Clearly, if we could show that all complemented ideals were completely complemented, then this implication would follow. However, in the non-amenable case we will show that $A(\mathbb F_N)$ has closed ideals that are complemented but not completely complemented (see Corollary \ref{corollary_prelim_1}). Whether or not this can happen in an amenable group is still not known.

We have already seen that even for $G = \mathbb{R}^3$ it is a very difficult task to identify which elements of $E\in \mathcal{R}_c(\hat{G})$ generate 
complemented ideals. That said, we do note that it is relatively easy to show that if $G$ is abelian and if $H$ is a 
closed subgroup of $G$, then $I_G(H)$ is complemented. In fact, it is completely
 complemented as in this case $A(G)$ is a \textit{MAX}-operator space. However, in \cite{For} the 
second author showed that if $G$ is the $ax+b$ - group, then $G$ has a closed subgroup $H$ for which $I_G(H)$ is not complemented in $A(G)$.  All told, we are certainly left to conclude that a complete classification of complemented ideals in the Fourier algebra of an arbitrary locally compact group $G$ would not be a reasonable goal. Instead, in this paper we will focus our attention on complemented and completely complemented ideals in the Fourier algebra of groups which contain the free group $\mathbb{F}_N$ as a closed subgroup. 

 Let $G$ be a discrete group. We will call a set $E\subseteq G$ a \emph{Leinert set} if the restriction algebra $A_G(E)$ is isomorphic to $\ell^2(E)$ as Banach algebras. We call $E$ a \emph{strong Leinert set} if $\ell^{\infty}(E) \subseteq M_{cb}A(G)$, the algebra of completely bounded multipliers of $A(G)$.

If  $E$ is any subset of a discrete group $G$, then it is well known that $\ell^2(E) \hookrightarrow A(G)$ contractively as spaces of functions on $G$.  Furthermore, by equipping $\ell^2(E)$ with various natural (i.e. row or column) Hilbertian operator space structures, this inclusion becomes a complete contraction.   For a locally compact group $G$ we call a subset $E$ \emph{uniformly discrete} if there is a neighbourhood $\mathcal{U}$ of the identity $e$ in $G$ such that if $g_1,g_2\in E$ with $g_1 \not = g_2$, then $g_1\mathcal{U} \cap g_2 \mathcal{U}=\emptyset$. If $E$ is a uniformly discrete subset of a locally compact group $G$, we show that there is a natural generalization of the embedding $\ell^2(E) \hookrightarrow A(G)$.  As an application of the ideas surrounding these objects, we obtain the following results on the complementation problem for ideals in  $A(G)$:  

Let $G$ be a locally compact group, $H$ a closed discrete subgroup of $G$,
and let $E$ be a Leinert set in $H$.  Then
\begin{itemize} 
\item[i)]The ideal $I_G(E)$ of functions in $A(G)$ vanishing on $E$
is always a complemented Banach subspace of $A(G)$. (See theorem \ref{main_result}.) 
\item[ii)]$I_G(E)$ is always an {\it invariantly weakly complemented} ideal
in $A(G)$. That is, $I_G(E)^\perp \subseteq VN(G)$ is complemented by a projection which is also an $A(G)$-bimodule map.  (See Proposition \ref{main_result_invariant})
\item[iii)] If $H$ is a noncommutative free group, then there are Leinert sets $E \subseteq H$ for which the ideal $I_G(E)$ is (weakly) complemented in $A(G)$, but not (weakly) completely complemented in $A(G)$.  (See Corollary \ref{failure_complete_complementation}.)
\item[iv)]  If $E$ is a strong Leinert set in $H$, then under certain conditions (in particular, when $G \in [SIN]_H$), $I_G(E)$ is complemented in $A(G)$ by a completely bounded projection.  (See Proposition \ref{prop-cb-proj}.) 
 \end{itemize}

\section{Preliminaries}

In this section we will outline some of the basic facts, notations and basic definitions we will need throughout the paper. We begin with some well known results on operator spaces  which we will need. 

\subsection{Operator Spaces}

If $\mathcal H $ is a Hilbert space, then we will denote by
$\mathcal H_r$ and $\mathcal H_c$ the row and column operator
Hilbert spaces over $\mathcal H$, respectively.  Refer to
\cite{Effros_Ruan} for the definitions and properties of these
operator spaces. If $X$ is any linear space, then we denote the
complex conjugate linear space of $X$ by $\overline{X}$.  The
elements of $\overline{X}$ will be denoted by $\overline{x}$ where
$x$ denotes some element of $X$.

Recall that if $X$ is an operator space, then $\overline{X}$ is
naturally an operator space by defining, for each $n \in \mathbb N$,
\begin{eqnarray*}
\|[\overline{x_{ij}}]\|_{M_n(\overline{X})} =\|[x_{ij}]\|_{M_n(X)}\qquad 
([\overline{x_{ij}}] \in M_n(\overline{X})).
\end{eqnarray*}

If $\mathcal H$ is any Hilbert space, then by \cite{Effros_Ruan} we
have the following completely isometric identifications:
$$(\mathcal H_c)^* = \overline{\mathcal H_r}, \ \ \ \ (\mathcal H_r)^* = \overline{\mathcal H_c}.$$
The dualities being given in both cases by the following dual
pairing
\begin{eqnarray*}
\langle \xi, \overline{\eta} \rangle = \langle \xi | \eta
\rangle_\mathcal H,
\end{eqnarray*}
where we take inner products $\langle \cdot | \cdot \rangle_H$ to be conjugate-linear in the second variable.

Recall that if $X$ and $Y$ are two operator spaces, then the
operator space projective tensor product of $X$ and $Y$ is denoted
by $X \widehat{\otimes} Y$ and defined by the family of norms
$\{\|\cdot\|_{n,\wedge}:M_n(X \otimes Y) \to \mathbb R_+\}_{n\in
\mathbb N}$ where
\[
\|u\|_{n,\wedge} = \inf\{\|\alpha\|\|x\|_p\|y\|_q\|\beta\|: \ u =
\alpha (x \otimes y)\beta \in M_n(X \otimes Y)\},
\]
where  $x \in M_p(X), \ y \in M_q(Y), \ \alpha \in M_{n,pq},$
and $\beta \in M_{pq, n}$.

\subsection{The Fourier Algebra and its Ideals}

Throughout this paper $G$ will denote a locally compact group with a fixed left Haar measure $dx$. We equip $L^1(G)$
with convolution and involution given by $f^*(x)=\Delta(x^{-1})\overline{f(x^{-1})}$, making it into an involutive Banach algebra which we call the \emph{group algebra of $G$}. The \emph{group $C^*$-algebra of $G$}, which we denote by $C^*(G)$ is simply the enveloping 
$C^*$-algebra of $L^1(G)$. 

By a representation of $G$ we will mean a homomorphism $\pi:G\to\mathcal{U}(\mathcal{H_{\pi}})$, where 
$\mathcal{U}(\mathcal{H_{\pi}})$ denotes the group of unitary operators on the Hilbert space $\mathcal{H_{\pi}}$. Given such a $\pi$, and
$\xi, \eta \in \mathcal{H}$ we call the function 
\[\pi_{\xi,\eta}:G \to \mathbb C; \qquad \pi_{\xi,\eta}(x)=\langle \pi(x) \xi | \eta \rangle_{\mathcal{H}_{\pi}}\]
a \emph{coefficient function} of $\pi$. We say that $\pi$ is continuous if each of its coefficient functions are continuous. 

We let 
\[B(G)= \{ u=\pi_{\xi,\eta} ~|~ \pi \textnormal{~is a continuous representation of $G$ and~} \xi,\eta \in \mathcal{H_{\pi}}\}.\]
Then $B(G)$ can be identified as the dual of $C^*(G)$ via the dual pairing 
\[\langle u, f \rangle = \int_G u(x) f(x) \,dx \]
for every $f\in L^1(G)$. With respect to the dual norm and the usual pointwise operations $B(G)$ becomes a commutative Banach algebra
which we call the \emph{Fourier-Sieltjes algebra} of $G$.  The set of all elements $u = \pi_{\xi,\xi} \in B(G)$ is denoted by $P(G)$.  $P(G)$ is precisely  the cone of continuous positive definite functions on $G$ and is identified with the space of all positive linear functionals on $C^*(G)$ under the identification $B(G) = C^*(G)^*$.

Amongst all continuous unitary representations of $G$, the most important would be  the \emph{left regular representation} $\lambda: G \to \mathcal{U}(L^2(G))$ which is defined by 
\[\lambda(x)(f)(y)=f(x^{-1}y)\]
for each $f\in L^2(G)$. We then define 
\[A(G)=\{u(x)=\lambda_{f,g}(x)~|~f,g\in L^2(G)\}.\]
$A(G)$ is a closed ideal of $B(G)$ which we call the \emph{Fourier algebra} of $G$. Its dual space is the \emph{group von Neumann algebra} $VN(G)$ given by 
\[VN(G) =\overline{span\{\lambda(x)~|~x\in G\}}^{WOT} \subseteq \mathcal{B}(L^2(G)).\]

We note that as the predual of a von Neumann algebra the Fourier algebra inherits a natural operator space structure under
 which it becomes a completely contractive Banach algebra.

We refer the reader to \cite{Eymard} for the basic properties of $A(G)$ and $B(G)$. 

\subsection{Complemented Ideals in Fourier Algebras}

 The Gelfand spectrum $\Delta(A(G))$ of $A(G)$ can be identified with $G$. If $E\subseteq G$ is closed, then we associate with $E$ the closed 
ideal of $A(G)$ given by 
\[I_G(E)=\{u \in A(G) ~|~ u(x)=0 \textnormal{~for every~} x\in E\}.\]

When there is no possibility of confusion, we will omit the
subscript $G$ and write $I(E)$ instead of $I_G(E)$.

\begin{definition} \label{complemented_ideal}
We say that a closed ideal $I$ is complemented in $A(G)$ if there exists a bounded linear map $P$ from $A(G)$ onto $I$ such that $P^2 = P$. We say that 
$I$ is invariantly complemented by $P$ if 
\[P(uv) =u\cdot P(v)\]
for every $u,v \in A(G)$.

We say that  a closed ideal $I$ is weakly complemented in $A(G)$ if there exists a bounded linear map $P$ from $VN(G)$ onto $I^{\perp}$ such that $P^2 = P$. We say that 
$I$ is invariantly weakly complemented by $P$ if 
\[P(u\cdot T) =u\cdot P(T)\]
for every $u \in A(G)$, $T \in VN(G)$, where we define  
\[\langle w, u\cdot T \rangle =\langle uw, T \rangle\]
for every $u,w \in A(G)$ and $T \in VN(G)$. 

We say that an ideal is (invariantly) completely complemented or (invariantly) completely weakly complemented if the implementing projection $P$ can be chosen to be completely bounded. 

\end{definition}

\begin{definition}
Let $G$ be a locally compact group and let $E \subseteq G$ be a
closed subset.  Denote by $A_G(E)$ the \textbf{restriction algebra}
\[
A_G(E) = \{\varphi:E \to \mathbb C: \ \varphi = u\big|_E \textrm{
for some $u \in A(G)$}\}.
\]
\end{definition}
When there is no possibility of confusion, we will simply write
$A(E)$ instead of $A_G(E)$. Observe that $A_G(E)$ identifies canonically with the quotient
algebra $A(G)/I_G(E)$. We equip $A_G(E)$ with the quotient norm
\[
\|\varphi\|_{A_G(E)} = \inf\{\|u\|_{A(G)}: u \in A(G) \ \textrm{ and
} \ u\big|_E = \varphi\}.
\]

\subsection{Mutipliers and Completely Bounded Multipliers of $A(G)$} 
 
A \emph{multiplier} of $A(G)$ 
is a (necessarily bounded and continuous) function 
$v \!: G \to \mathbb{C}$
such that $v A(G) \subseteq A(G)$. For each such  $v$ of $A(G)$, 
the linear operator $M_v$ on $A(G)$ defined by 
$M_{v}(u)=vu$ for each $u\in A(G)$ is bounded via the Closed 
Graph Theorem. The {\it multiplier algebra\/} 
of $A(G)$ is the closed subalgebra

\[
  MA(G) := \{ M_v : \textnormal{$v$ is a multiplier of $A(G)$} \}
\]
of $\mathcal{B} (A(G))$, where  $\mathcal{B} (A(G))$ denotes the algebra of all bounded linear 
operators from $A(G)$ to $A(G)$. Throughout this paper we will generally use $v$ in place of the 
operator $M_{v}$ and we will
write $\| v\|_{MA(G)}$ to represent the norm of $M_{v}$ in  $\mathcal{B} (A(G))$. \\

\noindent \textbf{Remark:}
Let $\mathbb{F}_N$ be the free group on $N$ generators $\{x_i\}_{i=1}^N$.  For  $g\in \mathbb{F}_N$ we  denote the usual word length of $g$ by 
 $|g|$ with respect to our $N$ generators. For the unit $e\in \mathbb{F}_N$, we have the convention that $|e| = 0$.

The length function and its connection to harmonic analysis on $\mathbb{F}_N$  has been extensively
studied by Haagerup in \cite{Haagerup}. In particular, the following result (\cite{Haagerup}  Lemma 1.7) will be
useful to us in the next section.

\begin{theorem}\label{Haagerup}
 Let $N\geq 2$ be a positive integer and let $\phi : \mathbb{F}_N \to \mathbb{C}$ be a function for
which
\[\sup\limits_{g\in \mathbb{F}_{N}} |\phi(g)|(1 + |g|)^2<\infty.\]
Then $\phi \in MA(\mathbb{F}_N)$  and
\[\|\phi\|_{MA(\mathbb{F}_N)}\leq 2 \sup\limits_{g\in \mathbb{F}_{N}} |\phi(g)|(1 + |g|)^2.\]
\end{theorem}

We would like to extend Theorem \ref{Haagerup}  to noncommutative free groups with infinitely
many generators. This can easily be done since the inequalities in Theorem \ref{Haagerup}  have no
dependence on the number $N$ of generators of $\mathbb{F}_N$. In particular, we have 

\begin{corollary}\label{Haagerup2} 
 Let $\mathbb{F}_I$ be a noncomutative free group on infinitely many generators
$\{x_i\}_{i\in I}$ and let $\phi : \mathbb{F}_I \to \mathbb{C}$ be a function for
which
\[\sup\limits_{g\in \mathbb{F}_{I}} |\phi(g)|(1 + |g|)^2<\infty.\]

Then $\phi \in MA(\mathbb{F}_I)$  and
\[\|\phi\|_{MA(\mathbb{F}_I)}\leq 2 \sup\limits_{g\in \mathbb{F}_{I}} |\phi(g)|(1 + |g|)^2.\]
\end{corollary}

\begin{proof}   Let $C_c(\mathbb{F}_I) )$  denote the algebra of finitely supported complex valued functions on $\mathbb{F}_I$.
Since $A(\mathbb{F}_I) ) = \overline{C_c(\mathbb{F}_I )}^{\|\cdot\|_{A(\mathbb{F}_I)}}$
and $\phi C_c(\mathbb{F}_I ) \subseteq C_c(\mathbb{F}_I)$, it suffices to show that
\[\|\phi u\|_{A(\mathbb{F}_I)}\leq 2 \sup\limits_{g\in \mathbb{F}_{I}} |\phi(g)|(1 + |g|)^2\| u\|_{A(\mathbb{F}_I)}\]
for every $u \in C_c(\mathbb{F}_I)$.

To see this, let $u \in C_c(\mathbb{F}_I)$ be arbitrary. Since the support of $u$ is finite, we can
find a finite subset $J\subseteq I$ such that $u$  is supported on the finitely generated subgroup
$\mathbb{F}_J =\langle x_i : i \in  J\rangle \leq \mathbb{F}_I$ . That is, $ u \in C_c(\mathbb{F}_J)$. Since the natural inclusion
$C_c(\mathbb{F}_J) \hookrightarrow C_c(\mathbb{F}_I)$ (given by extending by zero outside $\mathbb F_J$) 
extends to a complete isometry $ A(\mathbb{F}_J) \hookrightarrow A(\mathbb{F}_I )$, we get from Theorem \ref{Haagerup}:
\begin{eqnarray*}
\|\phi u\|_{A(\mathbb{F}_I)}&=& \|\phi u\|_{A(\mathbb{F}_J)}\\
&\leq& 2 \sup\limits_{g\in \mathbb{F}_{I}} |\phi(g)|(1 + |g|)^2\| u\|_{A(\mathbb{F}_J)} \\
&\leq& 2 \sup\limits_{g\in \mathbb{F}_{I}} |\phi(g)|(1 + |g|)^2\| u\|_{A(\mathbb{F}_I)} 
\end{eqnarray*}

\end{proof}

Since $A(G)$ is the predual of the von Neumann algebra $VN(G)$ it carries a natural operator structure
which makes $A(G)$ into a completely contractive Banach algebra. With this 
operator space structure we can define the 
$cb$\emph{-multiplier algebra} of $A(G)$ to be
\[
   {M}_{cb}A(G) := CB(A(G)) \cap {M}A(G),
\]
where $CB(A(G))$ denotes the completely contractive Banach algebra of all completely bounded 
linear maps from $A(G)$ into itself. 
We let $\| v\| _{M_{cb}A(G)}$ denote the $cb$-norm of 
the operator $M_{v}$. 
It is well known that ${M}_{cb}A(G)$ is a closed subalgebra 
of $CB(A(G))$ and 
is thus a (completely contractive) Banach algebra
with respect to the norm $\| \cdot \| _{M_{cb}A(G)}$.

It is known that in general, 
\[A(G)\subseteq B(G) \subseteq {M}_{cb}A(G) \subseteq {M}A(G)\] 
and that for $v \in B(G)$
\[\| v\| _{B(G)}\ge 
\| v\| _{M_{cb}A(G)}\ge \| v\| _{MA(G)}.\]
Moreover, if $v \in A(G)$, then 
\[\| v\| _{A(G)}=\| v\|_{B(G)}.\]
Furthermore, in case $G$ is an amenable group, we have 
\[B(G)={M}_{cb}A(G) = {M}A(G) \] 
and that 
\[\| v\| _{B(G)} = \| v\| _{M_{cb}A(G)} = \| v\| _{MA(G)}\]
for any $v\in B(G)$.

The following characterization of the $cb$-multipliers of $A(G)$ was given by Jolissaint in \cite{Jolissaint}: 

\begin{theorem}\label{Jolissaint} 
Let G be as above and let $\phi$ be a function on G. Then the
following conditions are equivalent:
\begin{itemize}
\item[i)] $\phi$ belongs to $M_{cb}A(G)$;
\item[ii)] there exist a Hilbert space $\mathcal{H}$ and bounded continuous functions $\xi, \eta$
from $G$ to $\mathcal{H}$ such that
\[\phi(t^{-1}s) = \langle \xi(s) | \eta(t)\rangle_{\mathcal{H}}\]
for all $s, t \in G$. 
\end{itemize} 

Moreover, if these conditions are satisfied, then
\[\|\phi\|_{M_{cb}A(G)}=\inf \|\xi\|_{\infty} \|\eta\|_{\infty} \]
where the infimum is taken over all pairs as in condition ii).
\end{theorem} 

We now finish this section by recalling the link between $cb$-multipliers of $A(G)$ and Schur multipliers on $\mathcal{B}(L^2(G))$.  In fact, since we will only require this correspondence in the case of discrete groups, we content ourselves with this special case.

\begin{definition}\label{Schur Multiplier} Let $G$ be a discrete group.  A function $\sigma:G \times G \to \mathbb C$ is called a Schur multiplier on $\mathcal{B}(\ell^2(G))$ if the infinite matrix 
\[S_{\sigma}T:=[\sigma(g,h)T(g,h)]_{(g,h)\in G \times G},\]
belongs to  $\mathcal{B}(\ell^2(G))$ for all $[T(g,h)]_{(g,h)\in G \times G}\in \mathcal{B}(\ell^2(G))$.
We denote the vector space of all Schur multipliers of $\mathcal{B}(\ell^2(G))$ by $V^{\infty}(G)$. Moreover, $V^{\infty}(G)$
becomes a pointwise Banach algebra of functions on $G\times G$ when $\sigma$ is  given the norm of $S_{\sigma}$ above as an operator on 
$\mathcal{B}(\ell^2(G))$.
\end{definition} 

In 1984, Bo$\dot{z}$jeko and Fendler established an isometric homomorphism between $M_{cb}A(G)$ and $V^{\infty}(G)$ \cite{Bozejko-Fendler}. 
We state the following  version of their result for discrete groups.

\begin{theorem}\label{Schur}
Let G be a discrete group, let $\mathcal{H}$ be a Hilbert space, and let $\xi, \eta:G \to \mathcal H$
be bounded functions. Then the function $\sigma_{\xi,\eta} : G \times G \to \mathbb{C}$  defined by
\[\sigma_{\xi,\eta}(g, h) = \langle \xi(h) |\eta(g)\rangle_{\mathcal{H}} \qquad ((g, h) \in G \times G)\]
belongs to $V^{\infty}(G)$, and furthermore $\|\sigma_{\xi,\eta}\|_{V^{\infty}(G)} \leq \|\xi\|_{\infty} \|\eta\|_{\infty}$. 

Conversely, every $\sigma \in V^{\infty}(G)$ is of the above form and 
\[
\|\sigma\|_{V^\infty(G)} = \inf \|\xi\|_\infty \|\eta\|_\infty,
\]
where the infimum is taken over all possible representations $\sigma = \sigma_{\xi,\eta}$ as above.
\end{theorem}

\section{Uniformly Discrete Subsets, Leinert Sets and Complemented Ideals}

In this section, we will see how Leinert sets and Strong Leinert sets generate complemented and completely complemented ideals 
respectively. In so doing we will exhibit an closed ideal in $A(\mathbb{F}_N)$ that is complemented but not even completely weakly complemented. We 
begin though by first looking at the properties of uniformly discrete subsets of $G$. 

\subsection{Uniformly Discrete Sets} 

\begin{definition} \label{uniformly_discrete}
Let $G$ be a locally compact group and let $E$ be a subset of $G$.
Then $E$ is said to be \textbf{uniformly discrete} in $G$ if there
exists some neighbourhood $\mathcal U$ of the identity $e \in G$
such that for all $g_1, g_2 \in E$
\begin{eqnarray*}
g_1\mathcal U \cap g_2 \mathcal U = \left\{ \begin{array}{ll}
g_1 \mathcal{U}, & \textrm{if $g_1 = g_2$}\\
\emptyset, & \textrm{if $g_1 \ne g_2$}
\end{array} \right.
\end{eqnarray*}
\end{definition}

\noindent \textbf{Observation:} It is a relatively easy task to show that $E$ is 
uniformly discrete if and only if there exists a neigbourhood $\mathcal{U}$ of $\{e\}$ 
such that if $g_1,g_2 \in E$ with  $g_1\not =g_2$, then $g_1^{-1}g_2 \not \in \mathcal{U}$. \\

\noindent \textbf{Examples:} If $G$ is a discrete group, then any
subset $E \subseteq G$ is uniformly discrete in $G$. Indeed, just
let $\mathcal U = \{e\}$ in the above definition.  More generally,
if $H$ is a discrete closed subgroup of a locally compact group $G$,
then any subset $E \subseteq H$ is a uniformly discrete subset of
$G$. To see this, just let $\mathcal U$ be any open neighbourhood of
$e$ for which $\mathcal U \cap H = \{e\}$. Clearly any finite subset of
a locally compact group $G$ is uniformly discrete in $G$ as well.

A more appropriate name for such sets would be \emph{left uniformly discrete}. However, we will 
not need to consider the right-sided analog in this paper. This does, however, lead us to consider the following question: \\

\noindent \textbf{Question:}  If $E \subseteq G$ is uniformly
discrete in $G$, then is $E^{-1}$ necessarily uniformly discrete in
$G$? \\

It is not hard to see that if $G$ is discrete or abelian, then the
answer is always yes. If $E \subseteq H$ for some closed discrete
subgroup $H \le G$, then the answer is yes as well. \\

We now show how the presence of uniformly discrete sets allows us to construct 
useful orthonormal families in $L^2(G)$.  Let $G$ be any locally compact group and let $E$ be some fixed uniformly discrete subset of $G$.  Let $\mathcal U$ be some open
neighbourhood of $e$ satisfying Definition \ref{uniformly_discrete}
for the set $E$.  Choose an symmetric neighbourhood $\mathcal V$ of $e$ such
that $\mathcal V^2 \subseteq \mathcal U$.  Now define a unit vector $\xi \in
L^2(G)$ by setting $$\xi = \frac{1_{\mathcal V}}{\|1_{\mathcal
V}\|_2},$$ and let $u \in P(G)\cap C_c(G)$ be defined by the
equation $$u = \overline{\xi}*\check{\xi} = \langle
\lambda(\cdot)\xi | \xi \rangle = \lambda_{\xi,\xi}.$$  It is easy to see that for the 
function $u$ we have $\textrm{supp}u \subseteq \mathcal U$.  Let us assume from now on that our function $u =
\overline{\xi}*\check{\xi}$ has been fixed.

\begin{lemma} \label{ONS}
The family $\{\lambda(h)\xi\}_{h \in E}$ forms an orthonormal system
in $L^2(G)$.
\end{lemma}

\begin{proof} For any $h \in E$ the left translation-invariance of Haar
measure gives $$\|\lambda(h)\xi\|_2 = \|\xi\|_2 = 1.$$ If $h_1,h_2
\in E$ are such that $h_2 \ne h_1$, then
\begin{eqnarray*}
\langle \lambda(h_1)\xi|\lambda(h_2)\xi \rangle &=& \int_G
\xi(h_1^{-1}g)\overline{\xi(h_2^{-1}g)}dg \\
&=& \frac{|h_1 \mathcal V \cap h_2 \mathcal V|}{|\mathcal V|}
\\
&\le& \frac{|h_1 \mathcal U \cap h_2 \mathcal U|}{|\mathcal V|} \\
&=& \frac{|\emptyset|}{|\mathcal V|} = 0 \ \ \ \ (\textrm{since $h_1 \ne h_2$}).
\end{eqnarray*} \end{proof} 

We now define a linear map $\Gamma_u:\ell^2(E) \to C_0(G) \cap
L^2(G)$ which is given by
\[
\Gamma_u \varphi = \sum_{h \in E} \varphi(h) (\delta_h*u),
\]
where $(\delta_h*u)(x) = u(h^{-1}x)$ and $(u*\delta_h) = u(xh)$ for $h,x \in G$.
The following Theorem shows that this map has some interesting properties:

\begin{theorem} \label{embedding_l2}
Let $E \subseteq G$ and $\Gamma_u$ be as above.  Then the
following are true: 
\begin{itemize} 
\item[i)] $(\Gamma_u \varphi) \big|_E = \varphi$ for all $\varphi \in
\ell^2(E).$ 
\item[ii)]$\Gamma_u(\ell^2(E)) \subseteq A(G)$ and the map
$\Gamma_u:\ell^2(E) \to A(G)$ is a contraction. 
\item[iii)] $\Gamma_u: \overline{\ell^2(E)_r} \to A(G)$ is a complete contraction. 
\end{itemize} 
\end{theorem}

\begin{proof} 
\begin{itemize}
\item[i)] By Lemma \ref{ONS}, given any $h_1 \ne h_2 \in E$ we
have $$u(h_2^{-1}h_1) = \langle \lambda(h_1)\xi | \lambda(h_2) \xi \rangle = 0
.$$
Consequently, if $s \in E$ and $\varphi \in \ell^2(E)$, then
\begin{eqnarray*}
(\Gamma_u \varphi)(s) &=& \sum_{h \in E} \varphi(h) u(h^{-1}s) \\
&=& \varphi(s)u(e) \\
&=& \varphi(s).
\end{eqnarray*}
That is, $\Gamma_u\varphi \big|_E = \varphi$ for all $\varphi
\in \ell^2(E)$. 
\item[ii)]  Let $\varphi \in \ell^2(E)$.  Then we have
\begin{eqnarray*}
(\Gamma_u \varphi)(x) &=& \sum_{h \in E} \varphi(h)u(h^{-1}x)  = \sum_{h \in E} \varphi(h) \langle \lambda(h^{-1}x)\xi | \xi \rangle \\
&=& \Big \langle \lambda(x)\xi \Big| \sum_{h \in E} \overline{\varphi(h)} \lambda(h)\xi \Big \rangle.
\end{eqnarray*}
Thus $$\Gamma_u\varphi = \lambda_{\xi, \sum_{h \in E} \overline{\varphi(h)} \lambda(h)\xi}\in A(G)$$ and
\begin{eqnarray*} \|\Gamma_u \varphi\|_{A(G)} &\le&
\|\xi\|_2\Big\|\sum_{h \in E}
\overline{\varphi(h)}\lambda(h)\xi\Big\|_2 = \|\overline{\varphi}\|_2 = \|\varphi\|_2.
\end{eqnarray*}
Therefore $\Gamma_u:\ell^2(E) \to A(G)$ is a contraction. 
\item[iii)] We will now show that $\Gamma_u:\overline{\ell^2(E)_r} \to A(G)$ is a complete contraction.  
Recall that the dual pairing 
\begin{eqnarray*}
\langle \overline{\eta}* \check{\xi}, T \rangle = \langle T \xi | \eta\rangle && (\overline{\eta} * \check{\xi} \in A(G), \ T \in VN(G))
\end{eqnarray*}
identifies $A(G)^*$ with $VN(G)$.  Now define a linear map $$q:\overline{L^2(G)_r} \widehat{\otimes} L^2(G)_c \to A(G)$$ by
$$q(\overline{\eta} \otimes \xi) = \overline{\eta}*\check{\xi}.$$  By \cite{Effros_Ruan} we have
$$(\overline{L^2(G)_r} \widehat{\otimes}L^2(G)_c)^* = \mathcal {CB}(L^2(G)_c) = \mathcal B(L^2(G))$$
completely isometrically.  Using this fact together with the above duality
between $A(G)$ and $VN(G)$, it follows that $q$ is nothing
other than the pre-adjoint of the canonical completely isometric inclusion $VN(G)
\hookrightarrow \mathcal B(L^2((G))$.  Consequently
$$q:\overline{L^2(G)_r} \widehat{\otimes}L^2(G)_c \to A(G)$$ is a
complete quotient map.

Now let $n \in \mathbb N$, $[\overline{\varphi_{ij}}] \in
M_n(\overline{\ell^2(E)_r})$, and note that
\begin{eqnarray*}
\Gamma_u^{(n)}([\overline{\varphi_{ij}}]) &=& \sum_{h \in
E}[\overline{\varphi_{ij}(h)}] \otimes (\delta_h*u) \\
&=&\sum_{h \in E}[\overline{\varphi_{ij}(h)}] \otimes
(\overline{\lambda(h)\xi}*\check{\xi}) \\
&=&(I_n \otimes q)\Big(\sum_{h \in E}[\overline{\varphi_{ij}(h)}]
\otimes \overline{\lambda(h)\xi} \otimes \xi \Big) \\
&=&(I_n \otimes q)\Big(\Big(\sum_{h \in
E}[\overline{\varphi_{ij}(h)}] \otimes \overline{\lambda(h)\xi}\Big)
\otimes \xi \Big).
\end{eqnarray*}
Therefore
\begin{eqnarray*}
\|\Gamma_u^{(n)}([\overline{\varphi_{ij}}])\|_{M_n(A(G))} &=&
\Big\|(I_n \otimes q)\Big(\Big(\sum_{h \in
E}[\overline{\varphi_{ij}(h)}] \otimes \overline{\lambda(h)\xi}\Big)
\otimes \xi \Big)\Big\|_{M_n(A(G))} \\
&\le& \Big\|\Big(\sum_{h \in E}[\overline{\varphi_{ij}(h)}] \otimes
\overline{\lambda(h)\xi}\Big) \otimes \xi \Big\|_{M_n \otimes
\overline{L^2(G)_r} \widehat{\otimes}L^2(G)_c} \\
&\le& \Big\|\sum_{h \in E}[\overline{\varphi_{ij}(h)}] \otimes
\overline{\lambda(h)\xi}\Big\|_{M_n(\overline{L^2(G)_r})}
\|\xi\|_{M_1(L^2(G)_c)} \\
&=& \Big\|\sum_{h \in E}[\varphi_{ij}(h)] \otimes
\lambda(h)\xi\Big\|_{M_n(L^2(G)_r)} \\
&=& \Big\| \sum_{h \in E}
[\varphi_{ij}(h)][\varphi_{ij}(h)]^*\Big\|^{1/2} \\
&=& \|[\varphi_{ij}]\|_{M_n(\ell^2(E)_r)} = \|[\overline{\varphi_{ij}}]\|_{M_n(\overline{\ell^2(E)_r})}. 
\end{eqnarray*}
Since $n \in \mathbb N$ and $[\overline{\varphi_{ij}}] \in
M_n(\overline{\ell^2(E)_r})$ were arbitrary,
$\Gamma_u:\overline{\ell^2(E)_r} \to A(G)$ is completely
contractive. 
\end{itemize}  \end{proof}

\noindent \textbf{Remark:}  Let $E \subseteq G$ and $u \in P(G) \cap C_c(G)$ be as above.  In addition, suppose that the family $\{\lambda(h^{-1})\xi\}_{h \in E}$ forms an ONS in $L^2(G)$.  In this case, define a new map $\check{\Gamma}_u:\ell^2(E) \to C_0(G) \cap
L^2(G)$ by setting
\[
\check{\Gamma}_u \varphi = \sum_{h \in E} \varphi(h) (u* \delta_{h^{-1}}).
\]
Then we obtain the following analogue of Theorem \ref{embedding_l2}:

\begin{theorem} \label{embedding_l2_column}
The map $\check{\Gamma}_u$ satisfies the following properties: 
\begin{itemize} 
\item[i)] $(\check{\Gamma}_u \varphi) \big|_E = \varphi$ for all $\varphi \in
\ell^2(E).$ 
\item[ii)] $\check{\Gamma}_u(\ell^2(E)) \subseteq A(G)$ and the map
$\check{\Gamma}_u:\ell^2(E) \to A(G)$ is a contraction. 
\item[iii)]$\check{\Gamma}_u: \ell^2(E)_c \to A(G)$ is a complete contraction. 
\end{itemize} 
\end{theorem}  

\begin{proof}
\begin{itemize} 
\item[i)] This is proved in a similar way to Theorem \ref{embedding_l2} Part $i)$, using the fact that $\{\lambda(h^{-1})\xi\}_{h \in E}$ forms an ONS in $L^2(G)$. 
\item[ii)]  Similar to the proof of Theorem \ref{embedding_l2} Part ii).  
\item[iii)]  We proceed as in the proof of Theorem \ref{embedding_l2}.  Let $n \in \mathbb N$ and let $[\varphi_{ij}] \in M_n(\ell^2(E)_c)$.  Then we have
\begin{eqnarray*}
\|\check{\Gamma}_u^{(n)}[\varphi_{ij}]\|_{M_n(A(G))} &=& \Big\|(I_n \otimes q)\Big(\overline{\xi} \otimes \Big(\sum_{h \in E}[\varphi_{ij}(h)] \otimes \lambda(h^{-1})\xi\Big) \Big) \Big\|_{M_n(A(G))} \\
&\le& \Big\|(\overline{\xi} \otimes \Big(\sum_{h \in E}[\varphi_{ij}(h)] \otimes \lambda(h^{-1})\xi\Big) \Big\|_{M_n \otimes \overline{L^2(G)_r} \widehat{\otimes} L^2(G)_c} \\
&\le&  \|\overline{\xi}\|_{\overline{L^2(G)_r}}\Big\| \sum_{h \in E}[\varphi_{ij}(h)] \otimes \lambda(h^{-1})\xi \Big\|_{M_n \otimes L^2(G)_c} \\
&=& \|[\varphi_{ij}]\|_{M_n(\ell^2(E)_c)}.
\end{eqnarray*}   
\end{itemize}     
\end{proof}

Given $T \in VN(G)$ and $v \in A(G)$, recall (see \cite{Eymard}) that we
can define a $VN(G)$-action on $A(G)$, denoted by $Tv$, by letting 
\begin{eqnarray*}
Tv(x) = \langle \delta_x *\check{v}, T \rangle \qquad (x \in G),
\end{eqnarray*}
where $\check{v}(x) = v(x^{-1})$ for any function $v:G \to \mathbb C$.  (Note that from \cite{Eymard} we have that $\|\check{v}\|_{A(G)} = \|v\|_{A(G)}$.)
It is not hard to see that when $v \in A(G) \cap L^2(G)$, $Tv$ is
nothing other than the image of the vector $v \in L^2(G)$ under the linear
operator $T \in VN(G) \subseteq \mathcal B (L^2(G))$.

Using the above module notation, we have the following corollary to
Theorem \ref{embedding_l2}:

\begin{corollary} \label{embedding_VN_l2}
Let $E \subseteq G$ be a uniformly discrete subset of $G$, let $u
\in P(G) \cap C_c(G)$ be as in Theorem \ref{embedding_l2}, and let
$n \in \mathbb N$.  Then for any matrix $[T_{ij}] \in M_n(VN(G))$ we
have
\begin{eqnarray*}
\Big\| \sum_{h \in E} [T_{ij}\check{u}(h)]^*[T_{ij}\check{u}(h)]
\Big\|^{1/2} \le \|[T_{ij}]\|_{M_n(VN(G))}.
\end{eqnarray*}
Suppose furthermore that the function $u$ satisfies 
\begin{eqnarray*}
u(h^{-1}xh) = u(x) \qquad (x \in G, \ h \in E).
\end{eqnarray*}
Then we have
\begin{eqnarray*}
\max \Big\{ \Big\| \sum_{h \in E}
[T_{ij}\check{u}(h)]^*[T_{ij}\check{u}(h)] \Big\|^{1/2}, \ \Big\|
\sum_{h \in E} [T_{ij}\check{u}(h)][T_{ij}\check{u}(h)]^* \Big\|^{1/2} \Big\} \le \|[T_{ij}]\|.
\end{eqnarray*}
\end{corollary}

\begin{proof} Let $T = [T_{ij}] \in M_n(VN(G))$, and let us consider the
quantity $$A(E,T): = \Big\| \sum_{h \in E}
[T_{ij}\check{u}(h)]^*[T_{ij}\check{u}(h)] \Big\|^{1/2}.$$ By the
very definition of the operator space $\ell^2(E)_c$, we have $$A(E,T) =
\|[T_{ij}\check{u}\big|_E]\|_{M_n(\ell^2(E)_c)}.$$  Since
$\ell^2(E)_c^* = \overline{\ell^2(E)_r}$ completely isometrically,
we have
\begin{eqnarray*}
A(E,T) &=& \sup_{[\overline{\varphi_{kl}}] \in
b_1(M_n(\overline{\ell^2(E)_r}))} \| \langle \langle
[T_{ij}\check{u}\big|_E], [\overline{\varphi_{kl}}]  \rangle
\rangle \| \\
&=& \sup_{[\overline{\varphi_{kl}}] \in
b_1(M_n(\overline{\ell^2(E)_r}))} \|[\langle T_{ij}\check{u}\big|_E,
\overline{\varphi_{kl}} \rangle]\|_{M_{n^2}} \\
&=& \sup_{[\overline{\varphi_{kl}}] \in
b_1(M_n(\overline{\ell^2(E)_r}))} \Big\|\Big[\sum_{h \in E}
T_{ij}\check{u}(h)\overline{\varphi_{kl}(h)}\Big]\Big\|_{M_{n^2}} \\
&=& \sup_{[\overline{\varphi_{kl}}] \in
b_1(M_n(\overline{\ell^2(E)_r}))} \Big\|\Big[\sum_{h \in E}
\overline{\varphi_{kl}(h)} \langle \delta_h*u, T_{ij} \rangle
\Big]\Big\|_{M_{n^2}} \\
&=& \sup_{[\overline{\varphi_{kl}}] \in
b_1(M_n(\overline{\ell^2(E)_r}))} \Big\|\Big \langle \Big \langle
\sum_{h \in E} [\overline{\varphi_{kl}(h)}] \otimes (\delta_h*u),
[T_{ij}] \Big \rangle
\Big \rangle \Big\|_{M_{n^2}} \\
&=& \sup_{[\overline{\varphi_{kl}}] \in
b_1(M_n(\overline{\ell^2(E)_r}))} \|\langle \langle
\Gamma_u^{(n)}[\overline{\varphi_{kl}}], [T_{ij}] \rangle
\rangle\| \\
&\le& \sup_{[\overline{\varphi_{kl}}] \in
b_1(M_n(\overline{\ell^2(E)_r}))}
\|\Gamma_u^{(n)}[\overline{\varphi_{kl}}]\|_{M_n(A(G))}
\|[T_{ij}]\|_{M_n(VN(G))} \\
&\le& \|[T_{ij}]\|_{M_n(VN(G))} \ \ \ \ (\textrm{by Theorem
\ref{embedding_l2} Part (3)}).
\end{eqnarray*}
Now suppose that $\delta_h*u*\delta_{h} = u$ for all $h \in E$.  Let
$[T_{ij}] \in M_n(VN(G))$ and consider the quantity $$B(E,T): =
\Big\|
\sum_{h \in E} [T_{ij}\check{u}(h)][T_{ij}\check{u}(h)]^* \Big\|^{1/2} .$$  We want to show that $$\max \{A(E,T), B(E,T)\} \le \|[T_{ij}]\|_{M_n(VN(G))}.$$   

Observe that for any $h \in E$, the fact that $u*\delta_h = \delta_{h^{-1}}*u$ implies that 
\begin{eqnarray*}
[T_{ij}\check{u}(h)]^* &=& [\overline{T_{ji}\check{u}(h)}]= [\overline{\langle u, \lambda(h^{-1})T_{ji} \rangle}]= [\langle u, T_{ji}^* \lambda(h) \rangle] \\
&=& [\langle u*\delta_h, T_{ji}^* \rangle]= [\langle \delta_{h^{-1}}*u, T_{ji}^* \rangle] \\
&=& [T_{ji}^*\check{u}(h^{-1})], 
\end{eqnarray*}
and consequently,
\begin{eqnarray*}
B(E,T)&=& \Big\| \sum_{h \in E} [T_{ji}^*\check{u}(h^{-1})]^*[T_{ji}^*\check{u}(h^{-1})] \Big\| = A({E^{-1},T^*}).
\end{eqnarray*}
Note that the condition on our function $u$ forces $E^{-1}$ to be uniformly discrete in $G$.  Furthermore we can apply Theorem \ref{embedding_l2} to the set $E^{-1}$ and the map $\tilde{\Gamma}_u:\overline{\ell^2(E^{-1})_r}) \to A(G)$ given by $\tilde{\Gamma}_u\varphi = \sum_{h \in E} \varphi(h)(\delta_{h^{-1}}*u)$ to deduce that $\|\tilde{\Gamma}_u\|_{cb} \le 1$.  This allows us to use the same argument that was used to bound $A(E,T)$ to get $A(E^{-1},T) \le \|[T_{ij}]^*\|_{M_n(VN(G))} = \|[T_{ij}]\|_{M_n(VN(G))}$.  The proof is now complete. 
\end{proof}

If we restrict our attention to discrete groups, Corollary \ref{embedding_VN_l2} can be viewed as a generalization of the well known fact that $VN(G)$ embeds completely contractively into both $\ell^2(E)_c$ and $\ell^2(E)_r$.  (See \cite{Pisier_Book} for a proof of this in the discrete case.)

\begin{corollary} \label{embedding_VN_l2_discrete}
Let $G$ be any discrete group and let $E \subseteq G$.  Then for any Hilbert space $\mathcal H$ and any finitely supported function $a:E \to \mathcal {B(H})$, we have  
\begin{eqnarray*}
\max\Big\{\Big\|\sum_{h \in E}a(h)^*a(h)\Big\|^{1/2},
\Big\|\sum_{h \in E}a(h)a(h)^*\Big\|^{1/2} \Big\} \le 
\Big\|\sum_{h \in E} a(h) \otimes \lambda(h)\Big\|_{\mathcal
{B(H)} \otimes_{min} VN(G)}.
\end{eqnarray*}
\end{corollary}

\begin{proof}
Applying Corollary \ref{embedding_VN_l2} to the set $E$ and the function $u = \delta_e$, we obtain the above result for any finite dimensional Hilbert space $\mathcal H$.  The finite dimensional case, however, is equivalent to the infinite dimensional case (see \cite{Pisier_Book} Chapter 2).   
\end{proof}

\subsection{Leinert Sets and Strong Leinert Sets in Discrete Groups}

Recall that if $X$ is any set, then the Hilbert sequence space
$\ell^2(X)$ is a Banach algebra under pointwise multiplication.

\begin{definition}
If $G$ is a discrete group and $E \subseteq G$, then $E$ is
called a \textbf{Leinert set} if $A_G(E) \cong \ell^2(E)$
isomorphically as Banach algebras.
\end{definition}

Recall again that for any discrete group $G$, $VN(G)$ embeds into
$\ell^2(G)$ by identifying $VN(G)$ with those functions $f \in
\ell^2(G)$ for which the left convolution operator
$\lambda(f):\ell^2(G) \to \ell^2(G)$ is bounded.  By duality, it can
be shown that a subset $E \subseteq G$ is a Leinert set if and only
the annihilator of the ideal $I(E) \subseteq A(G)$, $$I(E)^\perp =
\{\lambda(f) \in VN(G): \ \textrm{supp}f \subseteq E\},$$ is precisely all of $\ell^2(E)$  
(See \cite{Picardello}).

\begin{definition}\label{Leinert Property} 
Let $G$ be a discrete group. A set $E \subseteq G$  satisfies the Leinert condition if for all $n\in \mathbb{N}$ and 
for all $\{ x_i\}_{i=1}^{2n} \subseteq E$  
with $x_i \not = x_{i+1}$, we have that $x_1x_2^{-1}x_3x_4^{-1}\cdots x_{2n-1}x_{2n}^{-1}\not = e$.
\end{definition}

In \cite{Leinert}, Leinert showed that every set $E\subseteq G$ satisfying the Leinert condition is in fact a Leinert set. 

The following lemma, due to Bo$\dot{z}$ejko (\cite{Bozejko}), characterizes
Leinert sets in terms of the multiplier algebra $MA(G)$.

\begin{lemma}
Let $G$ be a discrete group and $E \subseteq G$.  Then
TFAE: \\ \\
$(1).$ $E$ is a Leinert set. \\
$(2).$ Every function in $\ell^\infty(E)$ belongs to $MA(G)$.
\end{lemma}

It follows from the above lemma that for any Leinert set $E$ in a
discrete group $G$, we have $$I(G \backslash E) = 1_E\cdot A(G) =
\ell^2(E).$$  Note however that the condition $I(G \backslash E) =
\ell^2(E)$ is not sufficient to ensure that $E$ is a Leinert set.  Indeed, it can
be shown that for \textit{any} discrete group $G$, there exists an infinite
set $E \subseteq G$ for which $I(G \backslash E) = \ell^2(E)$
(\cite{Picardello}).  

Clearly any finite subset of a discrete group is a Leinert set.  A typical example of an infinite Leinert set in the countably generated free group $\mathbb F_\infty$ on the generators $\{x_i\}_{i \in \mathbb N}$ is the set $E_n$ of all reduced words of length $n$ in $\mathbb F_\infty$.  For the Leinert set $E_1 = \{x_i, x_i^{-1}: \ i \in \mathbb N\}$, it can actually be shown that $\ell^\infty(E_1) \subseteq M_{cb}A(\mathbb F_\infty)$ \cite{Pisier}.  This example leads us to the next definition:  

\begin{definition}
Let $E$ be a subset of a discrete group $G$.  Then $E$ is called a
strong Leinert set if every function in $\ell^\infty(E)$ belongs to
$M_{cb}A(G)$.
\end{definition}

The following result, due to Pisier (\cite{Pisier} Proposition 3.2),
characterizes strong Leinert sets in terms of the operator space
structure of the subspace of $VN(G)$ consisting of those operators
supported on such sets.

\begin{proposition} \label{Pisier_prop}
Let $E$ be a subset of a discrete group $G$.  Then TFAE: \\ \\
$(1).$  $E$ is a strong Leinert set. \\
$(2).$  There exists some $C > 0$ such that for any Hilbert space
$\mathcal H$ and any finitely supported function $a:E \to \mathcal
{B(H)}$ we have
\begin{eqnarray*}
&&\Big\|\sum_{h \in E} a(h) \otimes \lambda(h)\Big\|_{\mathcal
{B(H)} \otimes_{min} VN(G)} \\
&\le&  C \max\Big\{\Big\|\sum_{h \in E}a(h)^*a(h)\Big\|^{1/2},
\Big\|\sum_{h \in E}a(h)a(h)^*\Big\|^{1/2} \Big\}.
\end{eqnarray*}
\end{proposition}

\noindent \textbf{Remark:}  As mentioned above, the prototypical example of an infinite strong Leinert set is given by the set $E_1 = \{x_i,x_i^{-1}\}_{i \in \mathbb N}\subset \mathbb F_\infty$ consisting of the free generators and their inverses. See for example \cite[Theorem 0.1]{Pisier}.\\

Next we will show that there exist Leinert sets $E$  in $\mathbb F_N$ ($2 \le N \le \infty$) which are not strong Leinert sets and that, in particular, $1_E\in MA(\mathbb{F}_{N} )\setminus   MA_{cb}(\mathbb{F}_{N})$. The existence of Leinert sets in free groups which are not strong Leinert sets was first established by Bo$\dot{z}$ejko in \cite{Bozejko}.  The following result gives an explicit example of such a set in $\mathbb F_\infty$ with the above properties. The case $N < \infty$ then follows immediately because $\mathbb F_\infty$ can be realized as a subgroup of $\mathbb F_N$. 

\begin{proposition}\label{Leinert} Let  $E = \{x_ix_j^{-1}: 1 \le i \le j < \infty\}$ where $S=\{x_i\}$ denotes a countable set of free generators of $\mathbb{F}_{\infty}$. 
Then the function $1_E$  belongs to $MA(\mathbb{F}_{\infty})$, but not $M_{cb}A(\mathbb{F}_{\infty})$.
\end{proposition}
\begin{proof} Observe that
\[\sup\limits_{g\in \mathbb{F}_{\infty}} |1_E(g)|(1 + |g|)^2 = \sup\limits_{g\in \mathbb{F}_{\infty}} (1 + |g|)^2 = (1 + 2)2 = 9,\]
so $1_E \in MA(\mathbb{F}_{\infty} )$  by Corollary \ref{Haagerup2}. Now let $\phi:= 1_E$  and suppose, to get a contradiction, that $\phi \in M_{cb}A(\mathbb{F}_{\infty})$. It then
follows from this assumption and Theorem \ref{Schur}  that the function $\sigma_{\phi}: \mathbb{F}_{\infty} \times  \mathbb{F}_{\infty} \to \mathbb{C}$
given by
\[\sigma_{\phi}(g, h) = \phi(gh^{-1}),\]
for all $(g, h) \in  \mathbb{F}_{\infty} \times  \mathbb{F}_{\infty}$ belongs to $V^{\infty} (\mathbb{F}_{\infty})$.

Let us now consider the associated Schur multiplier $S_{\sigma_{\phi}}$. Let $\{\delta_g : g \in \mathbb{F}_{\infty}\}$ denote
the canonical orthonormal basis for $\ell^2(\mathbb{F}_{\infty})$  and  identify $\mathcal{B}(\ell^2(S))$  with the corner $P\mathcal{B}(\ell^2(\mathbb{F}_{\infty}))P \subset \mathcal{B}(\ell^2(\mathbb{F}_{\infty}))$, where 
$P$ is the orthogonal projection from $\ell^2(\mathbb{F}_{\infty})$ onto the subspace $\ell^2(S)$. If 
$T = [T(x_i, x_j)]_{(i,j)\in \mathbb{N}\times \mathbb{N} }\in \mathcal{B}(\ell^2(S))$ , then $S_{\sigma_{\phi}}T$ is
given by the infinite matrix
\[S_{\sigma_{\phi}}T = [1_E(x_ix_j^{-1})T(x_i, x_j)]_{(i,j)\in \mathbb{N}\times \mathbb{N}}\]
where
\[
1_E(x_ix_j^{-1})= \left\{
\begin{array}{rl}
1 & \textnormal{if } i \leq j,\\
0 & \textnormal{if } i>j
\end{array} \right.
\]

Thus the map $T\to  S_{\sigma_{\phi}}T$ is just the upper-triangular truncation map on $ \mathcal{B}(\ell^2(S))$. Since
$\ell^2(S)$ is not finite dimensional, it follows that upper triangular truncation is not bounded
on  $ \mathcal{B}(\ell^2(S))$ (see for example Problems 8.15 and 8.16 in \cite{Paulsen}). The unboundedness of $S_{\sigma_{\phi}}$ contradicts the
fact that $\sigma_{\phi}\in V^{\infty}(\mathbb{F}_{\infty})$, and therefore 
we must have $\phi=1_E \in MA(\mathbb{F}_{\infty})\setminus M_{cb}A(\mathbb{F}_{\infty})$ .
\end{proof}

\begin{corollary} \label{prelim_1}
There exists Leinert sets $E \subset \mathbb F_\infty$ with $1_E
\in MA(\mathbb F_\infty)\backslash M_{cb}A(\mathbb F_\infty)$.
\end{corollary}

\begin{proof}  
Let  $E = \{x_ix_j^{-1}: 1 \le i \le j < \infty\}$ be as in the previous proposition. Then it is a routine calculation to show that $E$ satisfies the Leinert condition and is thus 
a Leinert set by \cite{Leinert}. The fact that $1_E
\in MA(\mathbb F_\infty)\backslash M_{cb}A(\mathbb F_\infty)$ is Proposition \ref{Leinert}.
\end{proof}

\begin{corollary} \label{corollary_prelim_1}
Let $E \subseteq \mathbb F_\infty$ be a Leinert set such that $1_E
\in MA(\mathbb F_\infty)\backslash M_{cb}A(\mathbb F_\infty)$. Then
the ideals $I(\mathbb F_\infty \backslash E)$ and $I(E)$ are
complemented in $A(\mathbb F_\infty)$, but not completely complemented in $A(\mathbb F_\infty)$. Furthermore, the annihilators $I(\mathbb F_\infty \backslash E)^\perp$ and $I(E)^\perp$ are complemented in
$VN(\mathbb F_\infty)$ but not completely complemented in $VN(\mathbb F_\infty)$.

In particular if $E = \{x_ix_j^{-1}: 1 \le i \le j < \infty\}$ where $S=\{x_i\}$ denotes the generators of $\mathbb{F}_{\infty}$, then $I(E)$ is complemented in $A(\mathbb{F}_{\infty})$ but is not   completely weakly complemented. 
\end{corollary}

\begin{proof} If we let $P:A(G) \to A(G)$ and $Q:A(G) \to A(G)$ be definded by $P(u)= 1_E u $ and $Q(u)=u-1_E u$. Then $P$ and $Q$ are projections onto $I(\mathbb F_\infty \backslash E)$ and $I(E)$ respectively.

Assume now that $E$ is such that  $I(E)^{\perp}$ is completely complemented in $VN(\mathbb{F}_{\infty})$. For a generic locally compact group $G$ we let $A_{cb}(G)$ be the closure of 
$A(G)$ when viewed as a subalgebra of $M_{cb}A(G)$. Then we know that $A(\mathbb{F}_N)$ is operator amenable \cite{F-R-S}. Since
 $\mathbb{F}_{\infty}$ can be realized as an open subgroup of $\mathbb{F}_N$, we get that  $A(\mathbb{F}_{\infty})$ is also operator  
amenable. It  follows from \cite{F-R-S}, Theorem 3.4 that $1_E \in M_{cb}A(\mathbb{F}_{\infty})$, which contradicts  our assumption that $1_E
\in MA(\mathbb F_\infty)\backslash M_{cb}A(\mathbb F_\infty)$. 

A similar argument as above shows that $I(\mathbb F_\infty \backslash E)^\perp$ is not completely 
complemented in $VN(\mathbb{F}_{\infty})$. 
 \end{proof}

\subsection{Complemented Ideals Vanishing on Leinert Sets in Discrete Subgroups of Locally Compact Groups}

Using our results on uniformly discrete subsets of locally compact groups, we will show (among other things) that if $H$ is a closed discrete subgroup of a locally compact group $G$ and $E \subseteq H$ is a Leinert set, then $I_G(E)$ is always complemented in $A(G)$ and always invariantly weakly complemented $A(G)$. \\

We begin with the following standard result: 

\begin{proposition} \label{extension_map}
Let $G$ be a locally compact group and let $E \subseteq G$ be a
closed subset.  Then the ideal $I(E) \subseteq A(G)$ has a Banach
space complement in $A(G)$ if and only if there exists a bounded
linear map $\Gamma:A(E) \to A(G)$ such that $\Gamma \varphi \big|_E
= \varphi$ for all $\varphi \in A(E)$.
\end{proposition}

\begin{proof} First suppose that $P:A(G) \to I(E)$ is a bounded projection.
Let $\epsilon > 0$ be fixed.  Given $\varphi \in A(E)$, let $u \in
A(G)$ be chosen so that $u\big|_E = \varphi$ and $\|u\|_{A(G)} \le
\|\varphi\|_{A(E)} + \epsilon$, and define
\begin{eqnarray*}
\Gamma \varphi = u - Pu \in A(G).
\end{eqnarray*}
Note that $\|\Gamma \varphi\|_{A(G)} \le
\|u\|_{A(G)}+\|P\|\|u\|_{A(G)} \le (1+\|P\|)\|\varphi\|_{A(E)} + (1
+ \|P\|)\epsilon $ and that $\Gamma \varphi \big|_E =  u\big|_E -
Pu\big|_E = \varphi$.  Also note that if $u_1 \in A(G)$ is any other
function such that $u_1\big|_E = \varphi$, then $u_1 - u \in I(E)$,
which implies that $u_1 - Pu_1 - (u - Pu) = u_1 - u - (P(u_1-u))  =
0$.  Thus $\Gamma:A(E) \to A(G)$ is a well defined map.  To see that
$\Gamma$ is linear, let $\varphi_1,\varphi_2 \in A(E)$ and let
$\alpha \in \mathbb C$.  Let $u_1,u_2 \in A(G)$ be two extensions of
$\varphi_1$ and $\varphi_2$ respectively.  Then clearly $\alpha u_1
+ u_2$ is an extension of $\alpha \varphi_1 + \varphi_2$, and
consequently
\[
\Gamma(\alpha \varphi_1 + \varphi_2) = \alpha u_1 + u_2 - P(\alpha
u_1 + u_2) = \alpha \Gamma \varphi_1 + \Gamma \varphi_2.
\]
Finally, since $\|\Gamma\| \le (1 + \|P\|)(1 + \epsilon) < \infty$,
$\Gamma$ is bounded.  Therefore $\Gamma$ is the required extension
map.

Conversely, suppose $\Gamma:A(E) \to A(G)$ is a bounded linear map
such that $\Gamma \varphi \big|_E = \varphi$ for all $\varphi \in
A(E)$.  For $u \in A(G)$ define $Pu = u - \Gamma(u\big|_E)$. Then
$P$ is obviously linear, and $\|Pu\|_{A(G)} \le (1 +
\|\Gamma\|)\|u\|_{A(G)}$.  Note that $Pu \big|_E = u\big|_E -
\Gamma(u\big|_E)\big|_E = u\big|_E - u\big|_E = 0$, so
$\textrm{ran}P \subseteq I(E)$.  Finally, if $u \in I(E)$, then
$u\big|_E = 0$, so $Pu = u$.  Therefore $P$ is the required
projection. \end{proof}

\begin{lemma} \label{restriction_transitivity}
Let $H \le G$ be a closed subgroup of a locally compact group $G$.  If $E \subseteq H$ is any closed subset, then $A_G(E) \cong A_H(E)$ completely isometrically.
\end{lemma}

\begin{proof} It follows from Herz's restriction theorem \cite{Herz} that $A(H) \cong A(G)/I_G(H)$ and $I_H(E) \cong I_G(E)/I_G(H)$ completely isometrically.  Consequently, 
\begin{eqnarray*}
A_H(E) &\cong& A(H)/I_H(E) \\
&\cong& (A(G)/I_G(H))/(I_G(E)/I_G(H)) \\
&\cong& A(G)/I_G(E) \\
&\cong& A_G(E)
\end{eqnarray*}
completely isometrically. \end{proof}

We are now in a position to state the main theorem of this section.

\begin{theorem} \label{main_result}
Let $G$ be a locally compact group and let $H$ be a discrete
subgroup of $G$.  If $E$ is a Leinert set in $H$, then the ideal
$I_G(E) \subseteq A(G)$ has a Banach space complement in $A(G)$.
\end{theorem}

\begin{proof}
By Proposition \ref{extension_map}, it suffices to find a
bounded linear extension map $\Gamma:A_G(E) \to A(G)$ such that
$\Gamma \varphi \big|_E = \varphi$ for all $\varphi \in A_G(E)$.

Since $E$ is a Leinert set in $H$, Lemma \ref{restriction_transitivity} implies that there exists some $C > 0$ for which $\|\varphi\|_{\ell^2(E)} \le C\|\varphi\|_{A_G(E)}$ for all $\varphi \in A_G(E)$.  Since $H$ is a discrete closed subgroup of $G$, $E$ is uniformly discrete in $G$.  Let $u \in P(G) \cap C_c(G)$ and $\Gamma_u:\ell^2(E) \to A(G)$ be given as in Theorem \ref{embedding_l2}.  Then $\Gamma_u: A_G(E) \to A(G)$ is bounded with norm $\le C$ and is the required extension map.   
\end{proof}

Theorem \ref{main_result} can be used to show that the Fourier algebra of any locally compact group containing a noncommutative free group as a discrete subgroup has (weakly) complemented ideals which fail to be (weakly) completely complemented.

\begin{corollary} \label{failure_complete_complementation}
Let $G$ be a locally compact group containing a noncommutative free
group as a discrete subgroup.  Then there are complemented ideals in
$A(G)$ which are complemented as Banach subspaces of $A(G)$, but not
as operator subspaces of $A(G)$.
\end{corollary}

\begin{proof} Since any noncommutative free group contains isomorphic copy
of $\mathbb F_\infty$ as a subgroup, $G$ therefore contains a copy
of $\mathbb F_\infty$ as a discrete subgroup. Let $E$ be a Leinert
set in $\mathbb  F_\infty$ satisfying the properties of Proposition
\ref{prelim_1} and consider the closed ideal $I_G(E) \subseteq
A(G)$.

By Theorem \ref{main_result}, $I_G(E)$ is complemented in $A(G)$ as
a Banach subspace.  To show that it is not complemented as an
operator subspace, assume by contradiction that there exists a
completely bounded projection $P:A(G) \to I_G(E)$.  Then, by
duality, the map $Q:=\textrm{id}_{VN(G)} - P^*$ must be a completely bounded
projection from $VN(G)$ onto $I_G(E)^\perp$.

Now, since $I_G(\mathbb F_\infty)^\perp$ and $VN(\mathbb F_\infty)$
are $\ast$-isomorphic as von Neumann algebras, and under this isomorphism
$I_G(E)^\perp$ and $I_{\mathbb F_\infty}(E)^\perp$ are identified,
we see that the restriction of $Q$ to $I_G(\mathbb F_\infty) \cong
VN(\mathbb F_\infty)$ yields the existence of a completely bounded
projection from $VN(\mathbb F_\infty)$ onto $I_{\mathbb
F_\infty}(E)^\perp$.  This, however, contradicts Corollary
\ref{corollary_prelim_1}. \end{proof}

\begin{corollary} \label{failure_complete_complementation2}
Let $G$ be a locally compact group such that every complemented ideal in $A(G)$ is completely complemented.  Then $G$ has an open amenable 
subgroup. In particular, if $G$ is almost connected, then $G$ is amenable. 
\end{corollary} 

\begin{proof} Let $G_e$ denote the connected component of the identity. If $G_e$ is not amenable, then $G_e$ contains $\mathbb{F}_2$ as discrete 
subgroup. However, this is impossible by Corollary \ref{failure_complete_complementation}.  As such we may conclude that $G_e$ is amenable. But $G$ 
has an open almost connected subgroup $H$. Since $H/G_e$ is compact and $G_e$ is amenable, $H$ is also amenable. \end{proof}

\noindent \textbf{Remarks:} 
\begin{itemize} 
\item[1)] We do not know if for every nonamenable group $G$ it is possible to find complemented ideals in $A(G)$ that are not completely
 complemented. Moreover, we do not know if there exists an amenable group $G$ for which $A(G)$ contains a complemented ideals that is not 
also completely complemented. In fact, it would be very desirable to show that no such amenable group exists. If we could establish this as a
 fact, then we could show, using operator amenability, that if $G$ is amenable and if $I(E)$ is complemented  in $A(G)$, then 
$E\in \mathcal{R}_c(G)$, the closed coset ring of $G$.  
\item[2)] We note that if $H$ is a closed subgroup of $G$ and if $E\subset H$ is such that $I_H(E)$ is complemented in $A(H)$, 
then it does not follow that $I_G(E)$ is complemented in $A(G)$ even if $E$ is uniformly discrete in $G$. To see an example of this we 
let $G$ be the $ax+b$ group. Then $G$ has a normal subgroup $H$ which is isomorphic to $\mathbb{R}$. In turn, $H$ has a discrete subgroup 
$H_1$ which is isomorphic to $\mathbb{Z}$. Now since $H$ is abelian, $I_H(H_1)$ is complemented in $A(H)$ (see  \cite{For}, Proposition 3.4). However, 
one can also follow the same reasoning as in \cite{For} Example 3.13 to show that $I_G(H_1)$ is not complemented in $A(G)$. 
Consequently, the assumption in Theorem \ref{main_result} that $E$ be a Leinert set is crucial. 
\end{itemize}

\subsection{Leinert Sets and Invariantly Weakly Complemented Ideals}

Let $G$ be a locally compact group and let $H$ be a discrete subgroup of $G$.  We will now show that whenever $E$ is a Leinert set in $H$, then the complemented ideal $I_G(E) \subseteq A(G)$ is always invariantly weakly complemented in $A(G)$.  This should be contrasted with the fact that for non-discrete $G$, $I_G(E)$ is never invariantly complemented in $A(G)$. 

We begin with a few preliminaries: Let $X$ and $Y$ be Banach spaces.  Recall that the space $\mathcal B(X, Y^*)$ is isometrically isomorphic to the dual space $(X \otimes^\gamma Y)^*$, where $\otimes^\gamma$ denotes the Banach space projective tensor product.  The duality is given by 
\begin{eqnarray*}
\langle x \otimes y, \Gamma \rangle = \langle y, \Gamma x \rangle && (x \in X, \ y \in Y, \ \Gamma \in \mathcal B(X, Y^*)).  
\end{eqnarray*} 
Given a Banach space $X$ and $C \ge 0$, we write $b_C(X) = \{x \in X: \|x\| \le C\}$.

\begin{proposition} \label{main_result_invariant}
Let $G$ be a locally compact group and let $H$ be a discrete
subgroup of $G$.  If $E$ is a Leinert set in $H$, then the ideal
$I_G(E) \subseteq A(G)$ is invariantly weakly complemented in $A(G)$.
\end{proposition}

\begin{proof} Given the Leinert set $E \subseteq H$, Lemma \ref{restriction_transitivity} tells us that there exists some $C > 0$ such that 
\begin{eqnarray*}
\|u\big|_E\|_2 \le C\|u\big|_E\|_{A_G(E)} \qquad (u \in A(G)).
\end{eqnarray*}
Since $H$ is a discrete subgroup of $G$, we can find an open neighborhood
$\mathcal V$ of the identity $e \in G$ such that $\mathcal V \cap H = \{e\}$.
Let $\{\mathcal V_\alpha\}$ be a neighbourhood basis at $e$ such that $\mathcal V_\alpha \subseteq \mathcal V$ for all $\alpha$.
For each $\alpha$, choose a function $u_\alpha \in P(G)\cap C_c(G)$ with $\textrm{supp} u_\alpha \subseteq \mathcal V_\alpha$.

From Theorem \ref{embedding_l2} and Theorem \ref{main_result}, we know that for 
each $\alpha$, the linear map $\Gamma_\alpha:A_G(E) \to A(G)$ given by
\begin{eqnarray*}
\Gamma_\alpha \varphi = \sum_{h \in E}\varphi(h)(\delta_h*u_\alpha) \qquad (\varphi \in A_G(E))
\end{eqnarray*}
is bounded with $\|\Gamma_\alpha\| \le C$, and satisfies $\Gamma_\alpha \varphi \big|_E = \varphi$ for all $\varphi
\in A_G(E)$.

Letting $$P_\alpha: = \Gamma_\alpha^*:VN(G) \to I_G(E)^\perp \subseteq VN(G),$$ we obtain a net of projections $$\{P_\alpha\} \subseteq b_{C}(\mathcal B(VN(G), VN(G))) = b_{C}((VN(G) \otimes^\gamma A(G))^*).$$  It is easy to see that for each $\alpha$, $P_\alpha$ is defined by the following equation:
\begin{eqnarray*}
P_\alpha T = \sum_{h \in E} \Big\langle u_\alpha, \lambda(h^{-1})T \Big\rangle \lambda (h) && (T \in VN(G)).
\end{eqnarray*}    
Note that the above sum converges in the $\ell^2$-sense for all $T \in VN(G)$.

Now consider the net $\{u_\alpha\} \subset A(G)$.  By passing to a subnet if necessary, we may assume that $u_\alpha \to m \in VN(G)^*$ weak$^\ast$, where $m$ is a topologically invariant mean on $VN(G)$. See \cite{Renaud}, Theorem 4. Since $\{P_\alpha\}$ is also a bounded net in $(VN(G) \otimes^\gamma A(G))^*$, the Banach-Alaoglu theorem implies that by possibly passing to yet another weak$^\ast$-convergent subnet, we may assume that $P = w^\ast-\lim_\alpha P_\alpha \in b_{C}(\mathcal B(VN(G), VN(G))$ exists.

We will now show that $P$ is an invariant projection onto $I_G(E)^\perp$.  To do this, first observe that $P(VN(G)) \subseteq I_G(E)^\perp$.  Indeed, suppose that $v \in I_G(E)$ and $T \in VN(G)$.  Then 
\begin{eqnarray*}
\langle v, PT \rangle &=& \langle T \otimes v, P \rangle = \lim_\alpha  \langle T \otimes v, P_\alpha \rangle \\
&=& \lim_\alpha \langle v, P_\alpha T \rangle = \lim_\alpha \sum_{h \in E} v(h) \Big\langle u_\alpha,  \lambda(h^{-1})T \Big\rangle \\ 
&=& \lim_\alpha 0 = 0.
\end{eqnarray*}   
Therefore $P(VN(G)) \subseteq I_G(E)^\perp$.  Next, we observe that $PT = T$ for all $T \in I_G(E)^\perp$.   If $T \in I_G(E)^ \perp$ then there exists some function $f \in \ell^2(E)$ such that $T = \sum_{h \in E} f(h)\lambda(h)$.  But then for every $\alpha$ we have  
\begin{eqnarray*} 
P_\alpha T &=& \sum_{h \in E} \Big\langle u_\alpha, \lambda(h^{-1})T \Big\rangle \lambda(h) \\
&=& \sum_{h \in E} \sum_{s \in E} f(s) u_\alpha(h^{-1}s) \lambda(h) \\
&=& \sum_{h \in E} f(h)\lambda(h) = T.
\end{eqnarray*}
From this it follows that $PT = T$.  We therefore have that $P:VN(G) \to I_G(E)^\perp$ is a bounded projection.  It remains to show that $P$ is invariant.  To do this, let $A_c(G) = A(G) \cap C_c(G)$, let $u \in A_c(G)$, $v \in A(G)$, and $T \in VN(G)$.  Then we compute 
\begin{eqnarray*}
\langle u, P(v \cdot T) \rangle &=& \lim_\alpha \langle u, P_\alpha (v \cdot T) \rangle \\
&=& \lim_\alpha \sum_{h \in E} u(h) \Big\langle u_\alpha, \lambda(h^{-1}) (v \cdot T) \Big\rangle \\
&=& \lim_\alpha \sum_{h \in E} u(h) \Big\langle v(\delta_h*u_\alpha), T \Big\rangle \\
&=& \lim_\alpha \sum_{h \in E} u(h) \Big\langle \delta_h *((\delta_{h^{-1}}*v)u_\alpha), T \Big\rangle \\
&=& \lim_\alpha \sum_{h \in E} u(h) \Big\langle u_\alpha, (\delta_{h^{-1}}*v) \cdot (\lambda(h^{-1})T) \Big\rangle \\
&=& \sum_{h \in E} u(h) \Big\langle (\delta_{h^{-1}}*v) \cdot (\lambda(h^{-1})T), m \Big\rangle \\
&=& \sum_{h \in E} u(h)v(h) \Big\langle \lambda(h^{-1})T, m \Big\rangle \ \ \ \ (\textrm{by the invariance of $m$}) \\
&=& \lim_\alpha \sum_{h \in E} u(h)v(h) \Big\langle u_\alpha, \lambda(h^{-1}) T \Big\rangle \\
&=& \lim_\alpha \langle uv, P_\alpha T \rangle \\
&=& \langle uv, PT \rangle = \langle u, v \cdot PT \rangle. 
\end{eqnarray*}
That is, $\langle u, P(v \cdot T) = \langle u, v \cdot PT \rangle$ for all $u \in A_c(G)$.  Since $A_c(G)$ is norm dense in $A(G)$, it follows that $P(v \cdot T) = v \cdot PT $.  \end{proof} 

\noindent \textbf{Remark:} Observe that the above invariant projection is given by the formula 
\begin{eqnarray*}
PT = \sum_{h \in E} \Big \langle \lambda(h^{-1})T, m \Big\rangle \lambda(h) && (T \in VN(G)),  
\end{eqnarray*}
where $m \in VN(G)^*$ denotes the invariant mean obtained above.  Note that since for any nondiscrete group $G$, $m$ must annihilate $C^*_\lambda(G)$ (\cite{Renaud}), it follows that $$C^*_\lambda(G) \subseteq \ker{P}.$$  

\subsection{Strong Leinert Sets and Completely Complemented Ideals}

Let $E \subseteq H \le G$ be as in the previous section.  A natural question that arises within the above framework is: {\it Under what conditions are the projections $$P, P_\alpha:VN(G) \to I_G(E)^\perp$$ constructed above completely bounded?} We know from Corollary \ref{failure_complete_complementation} that it can happen that {\it none} of the projections $P, P_\alpha: VN(G) \to I_G(E)^\perp$ are completely bounded.  On the other hand, it seems natural to expect that if $E$ is assumed to be a strong Leinert set in $H$, then $I_G(E)$ should be weakly completely complemented.  In this final section, we present some evidence and partial results in this direction. 

Let $G_d$ denote the abstract group $G$ equipped with its discrete topology, and let $\lambda_d:G \to \mathcal U(\ell^2(G))$ denote the left regular representation of $G_d$ on $\ell^2(G)$.  Let $C^*_{\lambda_d}(G_d)$ denote the reduced group C$^\ast$-algebra of $G_d$, and denote by $C^*_\delta(G)$ the norm closure of $\lambda(\ell^1(G))$ in $\mathcal B(L^2(G))$.  We may identify $C^*_\delta(G)$ with the C$^\ast$-algebra $C^*_{\lambda_0} \subseteq \mathcal B(L^2(G))$ generated by the representation $\lambda_0:G_d \to \mathcal U(L^2(G))$, where $\lambda_0(t) = \lambda(t)$ for all $t \in G$.  Observe that $C^*_\delta(G)$ is a C$^\ast$-subalgebra of $VN(G)$ and that $C^*_\delta(G)$ is $\sigma$-weakly dense in $VN(G)$. 

Our first goal is to show that if $E$ is a strong Leinert set, then the restriction $P\big|_{C^*_\delta(G)}:C^*_\delta(G) \to I_G(E)^\perp$ of the invariant projection $P$ constructed in Proposition \ref{main_result_invariant} is completely bounded.  To do this, we first need an elementary lemma.

\begin{lemma}
Consider the algebras $\lambda(\ell^1(G)) \subseteq \mathcal B(L^2(G))$ and $\lambda_d(\ell^1(G)) \subseteq  \mathcal B(\ell^2(G))$, and define $\pi:\lambda(\ell^1(G)) \to \lambda_d(\ell^2(G))$ by  
\begin{eqnarray*}
\pi(\lambda(f)) = \lambda_d(f) && (f \in \ell^1(G)).
\end{eqnarray*}
Then $\pi$ extends to a $\ast$-homomorphism from $C^*_\delta(G)$ onto $C^*_{\lambda_d}(G_d)$.  
\end{lemma}

\begin{proof}
Clearly $\pi$ is a $\ast$-homomorphism with dense range, therefore it suffices to show that $\pi$ is continuous from $C^*_{\delta}(G)$ to $C^*_{\lambda_d}(G_d)$.
 
Recall that for any locally compact group $G$, the left regular representation $\lambda_d$ is always weakly contained in $\lambda_0$ (\cite{Bedos}, Lemma 2).  Equivalently, this means that for any $u \in P(G_d)\cap A(G_d)$, there exists a net $\{u_\alpha\}$ of positive definite functions associated to $\lambda_0$ such that $u_\alpha \to u$ uniformly on compacta.  Fix $u \in P(G_d)\cap A(G_d)$ and $f \in \ell^1(G)$, and let $\{u_\alpha\}$ be such a net converging uniformly on compacta to $u$.  Then we have
\begin{eqnarray*}
\langle u , \lambda_d(f^**f) \rangle &=& \sum_{g \in G} (f^**f)(g)u(g) \\
&=& \lim_\alpha \sum_{g \in G} (f^**f)(g)u_\alpha(g) \\
&=& \lim_\alpha \langle u_\alpha , \lambda_0(f^**f) \rangle \\
&\le& \|\lambda_0(f^**f)\| = \|\lambda_0(f)\|^2 = \|\lambda(f)\|^2. 
\end{eqnarray*}
But this implies that
\begin{eqnarray*}
\|\pi(\lambda(f))\|^2 &=&\|\lambda_d(f)\|^2 \\
&=& \sup \{ \langle u, \lambda_d(f^**f) \rangle : u \in P(G_d)\cap A(G_d)\} \\
&\le& \|\lambda(f)\|^2.
\end{eqnarray*}
I.e. $\pi$ is continuous, and we are done. 
\end{proof} 

\begin{proposition}
If $E$ is a Leinert set in $H$ such that $1_E \in M_{cb}A(H)$ and $P:VN(G) \to I_G(E)^\perp$ is the invariant projection constructed in Proposition \ref{main_result}, then $P\big|_{C^*_\delta(G)}:C^*_\delta(G) \to I_G(E)^\perp$ is completely bounded.
\end{proposition}

\begin{proof}
First observe that since $P$ is invariant, we have 
\begin{eqnarray} \label{formula}
P\lambda(t) = \left\{ \begin{array}{ll}
\lambda(t), & \textrm{if $t \in E$}\\
0, & \textrm{if $t \notin E.$}
\end{array} \right.
\end{eqnarray}
Indeed, if $t \in E$, then $\lambda(t) \in I(E)^\perp$, so $P\lambda(t) = \lambda(t)$.  If $t \notin E$, then there exists some $u \in I_G(E)$ with $u(t)= 1$, and so for any $v \in A(G)$ we have $$\langle v, P(\lambda(t)) \rangle = \langle v, P(u \cdot \lambda(t)) \rangle = \langle v, u \cdot P(\lambda(t)) \rangle = \langle uv, P(\lambda(t)) \rangle = 0.$$  That is, $P\lambda(t) = 0$.

In addition, define a map $\tilde{P}:VN(G_d) \to I_{G_d}(E)^\perp$ by setting
\begin{eqnarray} \label{formula_2}
\tilde{P}\lambda_d(t) = \left\{ \begin{array}{ll}
\lambda_d(t), & \textrm{if $t \in E$}\\
0, & \textrm{if $t \notin E.$}
\end{array} \right.
\end{eqnarray} 
Since $1_E \in M_{cb}A(H) \subseteq M_{cb}A(G_d)$, $\tilde{P}$ is a well defined and completely bounded projection.

Now denote by $VN_H(G_d)$ the w$^\ast$-closure of $\lambda_d(\ell^1(H))$ in $VN(G_d)$ and denote by $VN_H(G)$ the w$^\ast$-closure of $\lambda(\ell^1(H))$ in $VN(G)$.  Observe that both $VN_H(G_d)$ and $VN_H(G)$ are von Neumann subalgebras of $VN(G_d)$ and $VN(G)$, respectively, and since $H$ is discrete, the map $\Phi:VN_H(G_d) \to VN_H(G)$ defined by  
\begin{eqnarray} \label{formula_3}
\Phi(\lambda_d(t)) = \lambda(t) && (t \in H),
\end{eqnarray}   
is a $\ast$-isomorphism.

Let $\pi:C^*_\delta(G) \to C^*_{\lambda_d}(G_d)$ be the canonical $\ast$-homomorphism defined at the beginning of this section.  Then using formulas (\ref{formula}), (\ref{formula_2}) and (\ref{formula_3}) it is easy to see that
\begin{eqnarray*}
P\lambda(t) = \Phi(\tilde{P}(\pi(\lambda(t)))) && (t \in G).
\end{eqnarray*}  
Extending by linearity and continuity, it follows that $$P \big|_{C^*_\delta(G)} = \Phi \circ \tilde{P} \circ \pi.$$
Since $\Phi$, $\tilde{P}$, and $\pi$ are all completely bounded, $P\big|_{C^*_\delta(G)}$ must be completely bounded.
\end{proof}

\noindent \textbf{Remark:}
The above result is somewhat unsatisfactory, since in most cases of interest, the projection $P:VN(G) \to I_G(E)^\perp$  is not normal, and we therefore cannot infer complete boundedness of $P$ from that of $P|_{C^*_\delta(G)}$ without further hypotheses.  One such additional hypothesis is presented in the following proposition.



\begin{proposition} \label{prop-cb-proj} Let $E$ be a strong Leinert set contained in a discrete subgroup $H$ of a locally compact group $G$. Let $u \in P(G) \cap C_c(G)$ and $\Gamma_u: A_G(E) \to A(G)$ be fixed as in Theorem \ref{embedding_l2} for the set $E$.  If, in addition, $\delta_h*u*\delta_h = u$ for all $h \in E$, then $I_G(E)$ is completely complemented in $A(G)$.  (In particular, if $G$ is a $[SIN]$-group, then $I_G(E)$ is always completely complemented in $A(G)$.
\end{proposition}

\begin{proof}
It suffices to show that the bounded map $\Gamma_u:A_G(E) \to A(G)$ is completely bounded, or equivalently, that the projection $$P_u = \Gamma_u^*: VN(G) \to I_G(E)^\perp,$$ given by
\begin{eqnarray*}
P_uT = \sum_{h \in E} T\check{u}(h)\lambda(h) && (T \in VN(G)), 
\end{eqnarray*} 
is completely bounded.

Since $E \subseteq H$ is a strong Leinert set, Proposition \ref{Pisier_prop} tells us that there is some $C > 0$ such that for any Hilbert space $\mathcal H$ and any finitely supported function $a:E \to \mathcal{B(H)}$ we have
\begin{eqnarray*}
&&\Big\|\sum_{h \in E} a(h) \otimes \lambda(h)\Big\|_{\mathcal
{B(H)} \otimes_{min} VN(G)} \\
&\le&  C \max\Big\{\Big\|\sum_{h \in E}a(h)^*a(h)\Big\|^{1/2},
\Big\|\sum_{h \in E}a(h)a(h)^*\Big\|^{1/2} \Big\}.
\end{eqnarray*}  
Let $n \in \mathbb N$ and $[T_{ij}] \in M_n(VN(G))$ be arbitrary.  Using the above inequality together with Corollary \ref{embedding_VN_l2}, we get 
\begin{eqnarray*}
\|P_u^{(n)}[T_{ij}]\|_{M_n(VN(G))} &=& \Big\|\sum_{h \in E} [T_{ij}\check{u}(h)]\otimes \lambda (h)\Big\|_{M_n \otimes VN(G)} \\
&\le& C \max\Big\{\Big\|\sum_{h \in E}[T_{ij}\check{u}(h)]^*[T_{ij}\check{u}(h)]\Big\|^{1/2},
\Big\|\sum_{h \in E}[T_{ij}\check{u}(h)][T_{ij}\check{u}(h)]^*\Big\|^{1/2} \Big\} \\
&\le& C \|[T_{ij}]\|_{VN(G)}.
\end{eqnarray*}
Thus $P_u$ is completely bounded, and hence $I_G(E)$ is completely complemented.

Finally, if $G$ is a $[SIN]$-group, note that for any neighbourhood $\mathcal U$ of the identity, one can find $u \in P(G) \cap C_c(G)$ such that $\delta_h*u*\delta_h = u$ for all $h \in G$.  Just choose such a $u$ with small enough support and consider the associated extension map $\Gamma_u$. 
\end{proof}

\end{document}